\documentclass[12pt]{article}

\usepackage{mathrsfs}
\usepackage{amsmath}
\usepackage{amssymb}
\usepackage{amsthm}
\usepackage{hyperref}
\usepackage{mathtools}
\usepackage{tikz-cd}
\usepackage{tikz}
\usetikzlibrary{arrows,decorations.markings}
\usepackage{caption}
\usepackage{bbm}
\usepackage{comment}

\usepackage{authblk}

\newtheorem{theorem}{Theorem}[section]
\newtheorem{lemma}[theorem]{Lemma}
\newtheorem{definition}[theorem]{Definition}
\newtheorem{Proposition}[theorem]{Proposition}
\newtheorem{remark}[theorem]{Remark}

\newcommand{\bZ}{\mathbb{Z}}
\newcommand{\bQ}{\mathbb{Q}}
\newcommand{\Zp}{\bZ_p}
\newcommand{\Qp}{\bQ_p}
\newcommand{\br}[1]{\overline{#1}}
\newcommand{\brQp}{\br{\bQ}_p}
\newcommand{\Fp}{\mathbb{F}_p}
\newcommand{\brFp}{\br{\mathbb{F}}_p}
\newcommand{\GL}{\mathrm{GL}}
\newcommand{\GLQp}{\GL_2(\Qp)}
\newcommand{\SLQp}{\mathrm{SL}_2(\Qp)}
\newcommand{\IZind}{\mathrm{ind}_{IZ}^G\>}
\newcommand{\KZind}{\mathrm{ind}_{KZ}^G\>}
\newcommand{\IZKZind}{\mathrm{ind}_{IZ}^{KZ}\>}

\newcommand{\im}{\mathrm{Im}\>}
\newcommand{\Ker}{\mathrm{Ker}\>}
\newcommand{\id}{\mathrm{id}}
\newcommand{\GQp}{\mathrm{Gal}(\brQp/\Qp)}
\newcommand{\ind}{\mathrm{ind}\>}
\newcommand{\matUp}{\beta\begin{pmatrix}
1 & \lambda \\
0 & 1
\end{pmatrix}w}
\newcommand{\rmss}{\mathrm{ss}}

\title{An Iwahori theoretic mod $p$ Local Langlands Correspondence}
\author{Anand Chitrao}
\affil{School of Mathematics, Tata Institute of Fundamental Research \\ Homi Bhabha Road, Mumbai - 400005, India.}

\begin{document}

\maketitle

\begin{abstract}
    We extend a comparison theorem of Anandavardhanan-Borisagar between the quotient of the induction of a mod $p$ character by the image of an Iwahori-Hecke operator and compact induction of a weight to the case of the trivial character. This involves studying the corresponding non-commutative Iwahori-Hecke algebra. We use this to give an Iwahori theoretic reformulation of the (semi-simple) mod $p$ Local Langlands Correspondence discovered by Breuil and reformulated functorially by Colmez. This version of the correspondence is expected to have applications to computing the mod $p$ reductions of semi-stable Galois representations.
\end{abstract}

\section{Notation and conventions}

    Let $p$ be a prime. Let $\omega$ and $\omega_2$ be fundamental characters of $I_{\Qp}$ of levels $1$ and $2$. The character $\omega$ has a canonical extension to $\GQp$. For an integer $c$ with $p + 1 \nmid c$, we choose an extension of $\omega_2^c$ to $\mathrm{Gal}(\brQp/\bQ_{p^2})$ so that the irreducible representation
    $\mathrm{ind}(\omega_2^c)$ obtained
    by inducing this extension from $\mathrm{Gal}(\brQp/\bQ_{p^2})$ to $\mathrm{Gal}(\brQp/\Qp)$ has determinant $\omega^c$. Let $\mu_{\lambda}$ be the unramified character of $\GQp$ sending a geometric Frobenius to $\lambda \in \brFp^*$.

    Let $G$ be the group $\GLQp$ and $Z$ be its center. Let $K$ be the maximal compact subgroup $\GL_2(\Zp)$. The Iwahori subgroup of upper triangular matrices mod $p$ is denoted by $I$.

    Let
    \[
        \alpha = \begin{pmatrix}1 & 0 \\ 0 & p\end{pmatrix}, \beta = \begin{pmatrix}0 & 1 \\ p & 0\end{pmatrix} \text{ and } w = \begin{pmatrix}0 & 1 \\ 1 & 0\end{pmatrix}.
    \]

    For $r \geq 0$, let $V_r$ be the symmetric $r$-th power of the standard representation $\Fp^2$ of $K$. Let $V_r^* \subset V_r$ be the $K$-subrepresentation of polynomials divisible by $X^pY - XY^p$. We think of these as representations of $KZ$ by making the scalar matrix $p$ act trivially. Let $d^r$ be the character $I \to \Fp^*$ that sends $\begin{pmatrix}a & b \\ c & d\end{pmatrix} \in I$ to $d^r \!\! \mod p$ and extend it to $IZ$ by sending $p \in Z$ to $1$.

    Let $\KZind \sigma$ denote the compact induction of a representation $\sigma$ of $KZ$ over $\Fp$. Functions in this space are denoted by $[g, v]$.
    Similarly, let $\IZind \sigma$ and $\IZKZind \sigma$ denote the compact induction and induction of a representation $\sigma$ of $IZ$, respectively. Functions in these spaces are denoted by $[[g, v]]$ and $[[[g, v]]]$, respectively. By spherical Hecke algebra, we mean the algebra $\mathrm{End}_G(\KZind \sigma)$ for an irreducible representation $\sigma$ of $KZ$. Similarly, by Iwahori-Hecke algebra, we mean $\mathrm{End}_G(\IZind \sigma)$.

    Let $[a] \in \{0, \ldots, p - 2\}$ be the class of an integer $a$ mod $p - 1$. Let $\delta_{a, b} = 1$ if $a = b$ and $0$ otherwise.

\section{Introduction}

Let $F$ be a local field of residue characteristic $p$. The study of smooth irreducible representations of $\GL_2(F)$ with a central character over fields of characteristic $p$ was initiated by Barthel and Livn\'e in \cite{BL95} and \cite{BL94}. In \cite{Bre03a}, Breuil studied the particularly tricky sub-class of such representations called the supersingular representations and proved that they are irreducible, at least when $F = \Qp$. He then wrote down the following definition, which he called the mod $p$ Local Langlands Correspondence (mod $p$ LLC), which is an injection from isomorphism classes of two-dimensional semi-simple mod $p$ representations of $\GQp$ to isomorphism classes of smooth representations of $\GLQp$.
\begin{definition}[{\cite[Definition 4.2.4]{Bre03a}}]\label{Breuil's mod p LLC}
    For $r \in \{0, \ldots, p - 1\}$, $\lambda \in \brFp$ and a smooth character $\eta : \Qp^* \to \brFp^*$, define
    \begin{enumerate}
        \item If $\lambda = 0$:
                \[
                    (\ind \omega_2^{r + 1}) \otimes \eta \xleftrightarrow{\quad \quad} \pi(r, 0, \eta)
                \]
        \item If $\lambda \neq 0$:
                \[
                    (\mu_{\lambda}\omega^{r + 1} \oplus \mu_{\lambda^{-1}})\otimes \eta \xleftrightarrow{\quad \quad} \pi(r, \lambda, \eta)^{\rmss} \oplus \pi([p - 3 - r], \lambda^{-1}, \eta \omega^{r + 1})^{\rmss}.
                \]
    \end{enumerate}
\end{definition}
\noindent More precisely, this is the semi-simple version of the mod $p$ LLC, which is all we consider in this paper. There is also a version without semi-simplification by Breuil \cite[Definition 2.2]{Bre10} and a functorial version due to Colmez \cite{Col10}.

One of the applications of this correspondence is to compute the reduction mod $p$ of irreducible two-dimensional crystalline representations of $\GQp$. Let us explain this idea now. For each $k \geq 2$ and $a_p \in \brQp$ of positive valuation, there exists an irreducible, two-dimensional crystalline representation $V_{k, a_p}$ of $\GQp$ over a $p$-adic number field. Let $B(V_{k, a_p})$ be the Banach space associated to $V_{k, a_p}$ under the $p$-adic Local Langlands Correspondence. In \cite{Bre03b}, Breuil conjectured that (the semi-simplification) of the reduction mod $p$ of $V_{k, a_p}$ corresponds to (the semi-simplification) of the reduction mod $p$ of $B(V_{k, a_p})$ under the bijection given in Definition \ref{Breuil's mod p LLC}. Using this conjecture, he computed the reduction mod $p$ of $V_{k, a_p}$ for $k \leq 2p$ ($k \neq 4$ if $p = 2$) and all $a_p$. The following famous diagram illustrates this idea:
\[\small
\begin{tikzcd}
    V_{k, a_p} \ar[d] \ar[r, leftrightarrow, "\text{$p$-adic LLC}"] &[2cm] B(V_{k, a_p}) \ar[d] \\
    \br{V}_{k, a_p} \ar[r, leftrightarrow, "\text{mod $p$ LLC}"] &[2cm] \> \br{B(V_{k, a_p})}.
\end{tikzcd}
\]
Breuil's conjecture was subsequently verified by Berger in \cite{Ber10}.

This idea of using the compatibility with respect to reduction mod $p$ between the $p$-adic and mod $p$ Local Langlands Correspondences was used in \cite{BG09} \cite{BG13}, \cite{GG15}, \cite{BG15}, \cite{BGR18}, \cite{GR23} to compute the reduction of $V_{k, a_p}$ for unbounded weights $k$ but bounded slopes $v_p(a_p)$. Thanks to results in these papers, the reduction of $V_{k, a_p}$ is known for arbitrary $k \geq 2$, $0 < v_p(a_p) < 2$ and $p$ odd (and in some cases $p \geq 5$). Nagel-Pande \cite{NP19} and Arsovski \cite{Ars21} also have some results for higher slopes. The growing complications in all these results were the reason Breuil only restricted to $k \leq 2p$ (see the comments after \cite[Th\'eor\`eme 1.4]{Bre03b}). The phenomenon of growing complications can be seen in \cite{Gha21} in the \emph{zig-zag} conjecture (now a theorem due to \cite{Gha22}) of Ghate, which describes the reduction mod $p$ of $V_{k, a_p}$ in one of the trickiest cases, namely for weights $k \equiv 2v_p(a_p) + 2 \mod p - 1$ and slopes $v_p(a_p)$ in the interval $(0, \frac{p - 1}{2})$. Recently Bhattacharya, Ghate and Vangala have computed $\br{V}_{k, a_p}$ in many cases for arbitrary $k \geq 2$ and $0 < v_p(a_p) < p$.

In this paper, we rewrite Breuil's mod $p$ LLC using the language of Iwahori induction:
\begin{theorem}[Iwahori mod $p$ LLC]\label{Iwahori mod p LLC}
    For $r \in \{0, \ldots, p - 1\}$, $\lambda \in \brFp$ and a smooth character $\eta: \Qp^* \to \brFp^*$, define
            
            \begin{itemize}
            \item If $\lambda = 0:$ 
                \[
                    (\ind \omega_2^{r + 1}) \otimes \eta \xleftrightarrow{\quad \quad} \frac{\IZind d^r}{(T_{-1, 0} + {\delta_{r, p - 1}T_{1, 0}}) + (T_{1, 2} + \delta_{r, 0}T_{1, 0})} \otimes \eta
                \]
            \item If $\lambda \neq 0:$
                \begin{eqnarray*}
                    (\mu_{\lambda}\omega^{r + 1} \oplus \mu_{\lambda^{-1}})\otimes \eta \!\!\!\! & \xleftrightarrow{\quad \quad} \!\!\!\! & \left(\frac{\IZind d^r}{(T_{-1, 0} + \delta_{r, p - 1}T_{1, 0}) + (T_{1, 2} + \delta_{r, 0}T_{1, 0} - \lambda)} \otimes \eta\right)^{\!\!\rmss} \!\!\oplus \\
                    && \left(\frac{\IZind d^{[p - 3 - r]}}{(T_{-1, 0} + \delta_{[p - 3 - r], p - 1}T_{1, 0}) + (T_{1, 2} + \delta_{[p - 3 - r], 0}T_{1, 0} - \lambda^{-1})}\otimes \eta\omega^{r + 1}\right)^{\!\!\rmss}\!\!\!.
                \end{eqnarray*}
            \end{itemize}
\end{theorem}

One of the applications of the Iwahori mod $p$ LLC is to compute the reduction of irreducible two-dimensional semi-stable non-crystalline representations of $\GQp$. The method to compute these reductions is analogous to the one discussed above for crystalline representations. Namely, the reduction of such a semi-stable representation is the one that corresponds, under Theorem \ref{Iwahori mod p LLC}, to the reduction of the Banach space associated to it by the $p$-adic LLC. This has been carried out recently by Ghate and the author for semi-stable representations with difference of Hodge-Tate weights lying in the range $[2, p]$ in a preprint \cite{CG23}. In particular, we have managed to compute the reduction mod $p$ of the semi-stable representation $V_{k, \mathcal{L}}$ for $k = p$ and $p + 1$ and arbitrary $\mathcal{L}$. To the best of our knowledge, this is the first time the compatibility between the $p$-adic and the mod $p$ LLC has been used to compute the reduction mod $p$ of semi-stable representations.

\section{Comparing Iwahori induction and spherical induction}\label{Comparison theorem section}
    In this section, we recall the Hecke algebra for the Iwahori induction of the mod $p$ character $d^r$. We also recall old comparison theorems and prove new ones between $\KZind V_{r}$ and $\IZind d^r$. The $0 < r < p - 1$ case is treated by Anandavardhanan and Borisagar in \cite{AB15}. We treat the cases $r = 0, \> p - 1$. This involves studying the corresponding non-commutative Iwahori-Hecke algebra.

\subsection{Recall of results in the commutative Hecke algebra case}

    It is well known (cf. \cite{BL95}) that if $0 < r < p - 1$, then the endomorphism algebra of the representation $\IZind d^r$ is the \emph{commutative Hecke algebra} generated by two operators $T_{-1, 0}$ and $T_{1, 2}$ satisfying
    \[
        T_{-1, 0}T_{1, 2} = 0 = T_{1, 2}T_{-1, 0}.
    \]
    These operators are defined by the following formulas (cf. \cite{AB15})
    \begin{eqnarray}\label{Formulas for T-10 and T12 in the commutative case}
        T_{-1, 0}[[\id, v]] = \sum_{\lambda \in I_1}\left[\left[\begin{pmatrix}p & \lambda \\ 0 & 1\end{pmatrix}, v\right]\right] \text{ and } T_{1, 2}[[\id, v]] = \sum_{\lambda \in I_1}\left[\left[\begin{pmatrix}1 & 0 \\ p\lambda & p\end{pmatrix}, v\right]\right],
    \end{eqnarray}
    where $v$ is a basis vector for the vector space underlying the representation $d^r$.

    Similarly, we have a Hecke algebra for representations induced from $KZ$ to $G$. The endomorphism algebra for the representation $\KZind V_r$ is generated by a single operator $T$. Its formula is given by
    \begin{eqnarray}\label{Spherical Hecke operator}
        T[\id, P(X, Y)] = \sum_{\lambda \in I_1}\left[\begin{pmatrix}p & \lambda \\ 0 & 1\end{pmatrix}, P(X, -\lambda X + pY)\right] + \left[\begin{pmatrix}1 & 0 \\ 0 & p\end{pmatrix}, P(pX, Y)\right].
    \end{eqnarray}
    
    The following results appear in \cite{AB15}:
    \begin{Proposition}\label{Kernel = Image in commutative Hecke algebra}
        We have
        \[
            \im T_{-1, 0} = \Ker T_{1, 2} \text{ and } \im T_{1, 2} = \Ker T_{-1, 0}.
        \]
    \end{Proposition}
    \begin{theorem}\label{Comparison theorem in commutative Hecke algebra}
        Let $0 < r < p - 1$. We have a $G$-equivariant isomorphism
        \[
            \KZind V_r \simeq \frac{\IZind d^r}{(T_{-1, 0})},
        \]
        which sends $[\id, x^r]$ to $[[\beta, 1]] + (T_{-1, 0})$. Under this isomorphism, the Hecke operator $T$ is mapped to $T_{1, 2}$.
    \end{theorem}
    \begin{theorem}\label{Flipping d and a in commutative Hecke algebra}
        For $0 < r < p - 1$, there is an isomorphism
        \[
            \IZind d^r \simeq \IZind a^r,
        \]
        which sends $[[\beta, 1]]$ to $[[\id, 1]]$. Under this isomorphism, $T_{-1, 0}$ is mapped to $T_{1, 2}$ and vice versa.
    \end{theorem}

    \begin{remark}
        Theorem~\ref{Comparison theorem in commutative Hecke algebra} was later generalized by Anandavardhanan and Jana in \cite{AJ21} and \cite{Jan23}.
    \end{remark}

\subsection{New results in the non-commutative Hecke algebra case}
    If $r = 0$, $p - 1$, then the endomorphism algebra of $\IZind d^r = \IZind 1$ is no longer only generated by $T_{1, 2}$ and $T_{-1, 0}$. Indeed,
    the na\"ive generalization of Theorem \ref{Flipping d and a in commutative Hecke algebra} suggests that there is an extra endomorphism of $\IZind 1$ given by $[[\id, 1]] \mapsto [[\beta, 1]]$. In fact, this operator exists and is denoted by $T_{1, 0}$. 
    It is proved in \cite{BL95} that the corresponding Hecke algebra is non-commutative and turns out to be generated by 
    $T_{1, 0}$ and $T_{1, 2}$ subject to the relations
    \begin{eqnarray}\label{Relations in the non-commutative Hecke algebra}
        T_{1, 0}^2 = 1 \text{ and } T_{1, 2}T_{1, 0}T_{1, 2} = -T_{1, 2}.
    \end{eqnarray}
    
    We write down explicit formulas for these Iwahori-Hecke operators acting on compactly supported functions on the Bruhat-Tits tree $\Delta$ of $\SLQp$. Recall that the vertices of $\Delta$ correspond to homothety classes of lattices in $\Qp^2$. There is a `canonical' vertex, namely the one represented by $\Zp^2$. Two vertices in $\Delta$ are joined by an edge in $\Delta$ if they are represented by lattices $L$ and $L'$ such that $pL \subsetneq L' \subsetneq L$. We state some simple lemmas without proof. In the next lemma, by the term `oriented edge', we mean an ordered pair of adjacent vertices in $\Delta$.
    \begin{lemma}
            The oriented edges of $\Delta$ are in a one-to-one correspondence with $G/IZ$.
    \end{lemma}
    We identify elements in the representation $\IZind 1$ with functions on the oriented edges of $\Delta$ with finite support: The element $[[g, 1]]$ is identified with the function that maps the edge 
    $g (\Zp^2, \alpha \Zp^2)$ to $1$ and all other edges to $0$. In the following lemma, $\br{e}$ denotes the edge $e$ with reverse orientation.
    \begin{lemma}\label{Small operator in the non-commutative Hecke algebra}
        If $e$ is the edge corresponding to the function $[[g, 1]]$, then $\br{e}$ is the edge corresponding to the function $[[g\beta, 1]]$.
    \end{lemma}
    In the next lemma, $o(e)$ and $t(e)$ denote the origin and terminal of an edge $e$, respectively.
    \begin{lemma}\label{Big operator in the non-commutative Hecke algebra}
        If $e$ is the edge corresponding to $[[g, 1]]$, then the edges $e'$ with $o(e') = t(e)$ and $e' \neq \br{e}$ correspond to $\left[\left[g \matUp, 1 \right]\right]$, where $\lambda \in I_1$.
    \end{lemma}
    \noindent Using \cite[Lemma 6]{BL95} and Lemmas \ref{Small operator in the non-commutative Hecke algebra}, \ref{Big operator in the non-commutative Hecke algebra} the formulas for $T_{1, 0}$ and $T_{1, 2}$ are given by
    \begin{eqnarray}
        T_{1, 0}[[\id, 1]] = [[\beta, 1]] \text{ and } T_{1, 2}[[\id, 1]] = \sum_{\lambda \in I_1}\left[\left[\begin{pmatrix}1 & 0 \\ p\lambda & p\end{pmatrix}, 1\right]\right].
    \end{eqnarray}
    \begin{remark}
    \begin{itemize}
    \item By Frobenius reciprocity, the Iwahori-Hecke algebra $\mathrm{End}_G(\IZind 1)$ is identified with the space $\mathrm{Hom}_{IZ}(1, \mathrm{Res}^{G}_{IZ}\IZind 1)$. Since the action of $G$ on functions in $\IZind 1$ is by right translations, this last space is isomorphic to functions in $\IZind 1$ that are right invariant under $IZ$. In other words, we can identify elements in this Iwahori-Hecke algebra with functions supported on a finite number of double cosets of $IZ$ in $G$. Under this identification, $T_{1, 0}$ and $T_{1, 2}$ correspond to $\mathbbm{1}_{IZ\beta I}$ and $\mathbbm{1}_{IZ\alpha^{-1}I}$, respectively, as explained in \cite{BL95}.
    
    \item The formula for $T_{1, 0}$ changes if we replace $\IZind 1$ by $\IZind \omega^s$ for any $0 \leq s \leq p - 2$ but the formula for $T_{1, 2}$ remains the same. Indeed, we have
    \begin{eqnarray}\label{Formulas for T10 and T12 in the non-commutative case}
        T_{1, 0}[[\id, 1]] = [[\beta, (-1)^{s}]] \text{ and } T_{1, 2}[[\id, 1]] = \sum_{\lambda \in I_1}\left[\left[\begin{pmatrix}1 & 0 \\ p\lambda & p\end{pmatrix}, 1\right]\right].
    \end{eqnarray}
    This follows from the projection formula which states that the map $\IZind \omega^s \to \omega^s \otimes (\IZind 1)$ sending $[[g, v]]$ to $gv \otimes [[g, 1]]$ is an isomorphism of $G$-representations.
    \end{itemize}
    \end{remark}

    The operator $T_{-1, 0}$ satisfies the relation (see \cite[Lemma 8(4)]{BL95})
    \begin{eqnarray}\label{T-10 in the non-commutative case}
        T_{-1, 0} = T_{1, 0}T_{1, 2}T_{1, 0}.
    \end{eqnarray}
    The formula for $T_{-1, 0}$ continues to be given by \eqref{Formulas for T-10 and T12 in the commutative case}.

    We now extend the results of \cite{AB15} to this non-commutative setting. We will use the following classical lemma from, e.g., \cite{GJ23}:
    \begin{lemma}[{\cite[Theorem 1.1]{GJ23}}]\label{The psi function}
        For $r \geq p + 1$, the map
        \[
            \psi : V_r \to \ind_{\mathrm{B}(\Fp)}^{\GL_2(\Fp)} d^r
        \]
        that sends a polynomial $P \in V_r$ to the function $\psi_P : \GL_2(\Fp) \to \Fp$ defined by
        \[
            \psi_P\left(\begin{pmatrix}a & b \\ c & d\end{pmatrix}\right) = P(c, d)
        \]
        induces an isomorphism of $\GL_2(\Fp)$-representations
        \[
            V_r/V_r^* \xrightarrow{\sim} \ind_{\mathrm{B}(\Fp)}^{\GL_2(\Fp)} d^r.
        \]
    \end{lemma}
    \noindent We identify the $KZ$-representation obtained by inflating the $\GL_2(\Fp)$-representation $\ind_{\mathrm{B}(\Fp)}^{\GL_2(\Fp)} d^r$ along the reduction mod $p$ map $K \to \GL_2(\Fp)$ (and making the scalar matrix $p$ act trivially) with $\IZKZind d^r$ since the subgroup $I$ reduces mod $p$ to $\mathrm{B}(\Fp)$.
    
    We first prove an analog of Proposition \ref{Kernel = Image in commutative Hecke algebra}.
    \begin{Proposition}
        We have
        \[
            \Ker T_{1, 2}T_{1, 0} = \im (1 + T_{1, 2}T_{1, 0}) \text{ and } \Ker (1 + T_{1, 2}T_{1, 0}) = \im T_{1, 2}T_{1, 0}.
        \]
    \end{Proposition}
    \begin{proof}
        Note that the second identity in \eqref{Relations in the non-commutative Hecke algebra} implies that the operator $-T_{1, 2}T_{1, 0}$ is an idempotent. This proposition then follows from the theory of idempotent endomorphisms.
    \end{proof}
    The following remark is the analog of Theorem \ref{Flipping d and a in commutative Hecke algebra}.
    \begin{remark}\label{Order 2 automorphism}
        The map $T_{1, 0}$ induces an isomorphism $\IZind 1 \to \IZind 1$. Under this map, the operator $T_{1, 2}$ is mapped to $T_{-1, 0}$ and vice versa. This is because of relation \eqref{T-10 in the non-commutative case} and $T_{1, 0}^2 = 1$.
    \end{remark}

    Before proving the analog of Theorem \ref{Comparison theorem in commutative Hecke algebra}, we state and prove a lemma.
    \begin{lemma}\label{Direct sum of V0 and Vp-1}
        The representation $V_{2p - 2}/V_{2p - 2}^*$ splits as a direct sum $V_{p - 1} \oplus V_0$. The copies of $V_{p - 1}$ and $V_{0}$ in $V_{2p - 2}/V_{2p - 2}^*$ are the subspaces generated by $X^{2p - 2}$ and $X^{2p - 2} - X^{p - 1}Y^{p - 1} + Y^{2p - 2}$, respectively.
    \end{lemma}
    \begin{proof}
        By \cite[Lemme 5.1.3 (ii)]{Bre03b}, we have the following exact sequence
        \[
            0 \to V_{p - 1} \to V_{2p - 2}/V_{2p - 2}^* \to V_{0} \to 0,
        \]
        which we think of as an exact sequence of $\GL_2(\Fp)$-representations.
        Since $V_{p - 1}$ is a projective, and hence an injective $\GL_2(\Fp)$-representation, this sequence splits. The injection of $V_{p - 1}$ into $V_{2p - 2}/V_{2p - 2}^*$ is given by sending the generator $X^{p - 1}$ to $X^{2p - 2}$ by \cite[Lemme 5.1.3 (ii)]{Bre03b}. Next, note that there is a unique copy of $V_0$ in $V_{2p - 2}/V_{2p - 2}^*$. Using Lemma \ref{The psi function}, we see that $$V_{2p - 2}/V_{2p - 2}^* \simeq \mathrm{ind}_{B(\Fp)}^{\GL_2(\Fp)} 1.$$ Clearly, the copy of $V_0$ in $\mathrm{ind}_{B(\Fp)}^{\GL_2(\Fp)} 1$ is given by the constant functions. A short computation shows that under the explicit isomorphism displayed above, the polynomial $X^{2p - 2} - X^{p - 1}Y^{p - 1} + Y^{2p - 2}$ maps to the constant function $1$.
    \end{proof}
    
    \begin{theorem}\label{Comparison theorem in the non-commutative Hecke algebra case}
        The map
        \begin{eqnarray*}
            \IZind 1 & \to & \KZind V_{2p - 2}/V_{2p - 2}^* \\
            \left. [[\id, 1]] \right. & \mapsto & [\id, Y^{2p - 2} - X^{p - 1}Y^{p - 1}]
        \end{eqnarray*}
        induces isomorphisms
        \begin{eqnarray*}
            \frac{\IZind 1}{(T_{1, 2}T_{1, 0})} \simeq \KZind V_0 & \text{ and } & \frac{\IZind 1}{(1 + T_{1, 2}T_{1, 0})} \simeq \KZind V_{p - 1}.
        \end{eqnarray*}
            Moreover, the operators $T_{-1, 0} + T_{1, 0}$ and $T_{-1, 0}$ on the left
            correspond to the operator $T$ on 
            the right under the isomorphisms mentioned above.
    \end{theorem}
        \begin{proof}
            We first prove that the map
            \begin{eqnarray*}
                \IZind 1 & \to & \KZind V_{2p - 2}/V_{2p - 2}^* \\
                \left. [[\id, 1]] \right. & \mapsto & [\id, Y^{2p - 2} - X^{p - 1}Y^{p - 1}]
            \end{eqnarray*}
            is an isomorphism. Since $IZ \subseteq KZ \subseteq G$ with $IZ \subseteq KZ$ of finite index, $\IZind 1$ is isomorphic to $\KZind \IZKZind 1$ by sending
            \[
                [[\id, 1]] \mapsto \big[\id, [[[\id, 1]]]\big],
            \]
            where single brackets denote functions in $\KZind 1$, double brackets denote functions in $\IZind 1$ and triple brackets denote functions in $\IZKZind 1$. A short computation using Lemma~\ref{The psi function} shows that the map sending $Y^{2p - 2} - X^{p - 1}Y^{p - 1}$ to $[[[\id, 1]]]$ is an isomorphism between $V_{2p - 2}/V_{2p - 2}^*$ and $\IZKZind 1$.
        
            Using the relation $T_{1, 2}T_{1, 0}T_{1, 2} = -T_{1, 2}$ in the non-commutative Hecke algebra for $\IZind 1$, we see that the operator $-T_{1, 2}T_{1, 0}$ is an idempotent. Therefore we get a splitting 
            \[
                \IZind 1 = \im (-T_{1, 2}T_{1, 0}) \oplus \im (1 + T_{1, 2}T_{1, 0}).
            \]
            We also know from Lemma~\ref{Direct sum of V0 and Vp-1} that
            \[
                \KZind V_{2p - 2}/V_{2p - 2}^* = \KZind V_{p - 1} \oplus \KZind V_{0},
            \]
            where $\KZind V_{p - 1}$ is identified with the subspace generated by $[\id, X^{2p - 2}]$ and $\KZind V_{0}$ is identified with the subspace generated by  $[\id, X^{2p - 2} - X^{p - 1}Y^{p - 1} + Y^{2p - 2}]$. Therefore to prove the first part of this theorem, we need to show that
            \begin{itemize}
                \item $\im T_{1, 2}T_{1, 0} \mapsto \KZind V_{p - 1}$,
                \item $\im (1 + T_{1, 2}T_{1, 0}) \mapsto \KZind V_{0}.$
            \end{itemize}
            Indeed, under the map $\IZind 1 \to \KZind V_{2p - 2}/V_{2p - 2}^*$, the element $T_{1, 2}[[\id, 1]]$ maps to
            \[
                \sum_{\lambda \in I_1}\left[\begin{pmatrix}1 & 0 \\ p\lambda & p\end{pmatrix}, Y^{2p - 2} - X^{p - 1}Y^{p - 1}\right] = \sum_{\lambda \in I_1}\left[\beta\begin{pmatrix}1 & \lambda \\ 0 & 1\end{pmatrix}w, Y^{2p - 2} - X^{p - 1}Y^{p - 1}\right].
            \]
            Since $\begin{pmatrix}1 & \lambda \\ 0 & 1\end{pmatrix}w \in KZ$, we transfer it to the other side to get
            \[
                \sum_{\lambda \in I_1}[\beta, X^{2p - 2} - (\lambda X + Y)^{p - 1}X^{p - 1}].
            \]
            Expanding $(\lambda X + Y)^{p - 1}$ and summing over $\lambda$ using 
            the identity 
            \begin{eqnarray}\label{sum of powers of roots of unity}
            \sum_{i = 0}^{p - 1}[i]^j = 
            \begin{cases}
                0 & \text{ if } p - 1 \nmid j \\
                p - 1 & \text{ if } p - 1 \mid j,
                \end{cases}
                & \text{ for $j \geq 1$,}
            \end{eqnarray}
            we see that
            \begin{eqnarray}\label{Image of T12 maps to}
                T_{1, 2}[[\id, 1]] \mapsto [\beta, X^{2p - 2}].
            \end{eqnarray}
            \begin{itemize}
            \item Using \eqref{Image of T12 maps to}, we see that
            \[
                T_{1, 2}T_{1, 0}[[\id, 1]] = T_{1, 2}[[\beta, 1]] \mapsto [\id, X^{2p - 2}].
            \]
            Moreover, $[\id, X^{2p - 2}]$ generates $\KZind V_{p - 1}$ in $\KZind V_{2p - 2}/V_{2p - 2}^*$.

            \item Again using \eqref{Image of T12 maps to}, we see that
            \[
                (1 + T_{1, 2}T_{1, 0})[[\id, 1]] = [[\id, 1]] + T_{1, 2}[[\beta, 1]] \mapsto [\id, X^{2p - 2} - X^{p - 1}Y^{p - 1} + Y^{2p - 2}].
            \]
            Moreover $[\id, X^{2p - 2} - X^{p - 1}Y^{p - 1} + Y^{2p - 2}]$ generates $\KZind V_{0}$.
            \end{itemize}

            We have proved that the natural map induces isomorphisms
            \begin{eqnarray}\label{Main comparison theorems in the non-commutative case}
                \frac{\IZind 1}{(T_{1, 2}T_{1, 0})} \xrightarrow{\sim} \KZind V_0 \quad \text{ and } \quad \frac{\IZind 1}{(1 + T_{1, 2}T_{1, 0})} \xrightarrow{\sim} \KZind V_{p - 1}.
            \end{eqnarray}
            These maps are respectively given by
            \[
                [[\id, 1]] \mapsto [\id, 1] \quad \text{ and } \quad [[\id, 1]] \mapsto [\id, -X^{p - 1}].
            \]
            Indeed, the first map sends $[[\id, 1]]$ to $[\id, Y^{2p - 2} - X^{p - 1}Y^{p - 1}] \equiv [\id, X^{2p - 2} - X^{p - 1}Y^{p - 1} + Y^{2p - 2}] \mod \KZind V_{p - 1}$ and the second map sends $[[\id, 1]]$ to $[\id, Y^{2p - 2} - X^{p - 1}Y^{p - 1}] \equiv [\id, -X^{2p - 2}] \mod \KZind V_0$. Now using \eqref{T-10 in the non-commutative case} and \eqref{Relations in the non-commutative Hecke algebra}, we see that $\im T_{1, 2}T_{1, 0} \subseteq \Ker (T_{-1, 0} + T_{1, 0})$ and $\im (1 + T_{1, 2}T_{1, 0}) \subseteq \Ker T_{-1, 0}$. Therefore they are operators on the quotients
            \[
                \frac{\IZind 1}{(T_{1, 2}T_{1, 0})} \quad \text{ and } \quad \frac{\IZind 1}{(1 + T_{1, 2}T_{1, 0})},
            \]
            respectively. Then under the first map in \eqref{Main comparison theorems in the non-commutative case}, we see that
            \[
                (T_{-1, 0} + T_{1, 0})[[\id, 1]] = \sum_{\lambda \in I_1}\left[\left[\begin{pmatrix}p & \lambda \\ 0 & 1\end{pmatrix}, 1\right]\right] + [[\beta, 1]] \mapsto \sum_{\lambda \in I_1}\left[\begin{pmatrix}p & \lambda \\ 0 & 1\end{pmatrix}, 1\right] + \left[\alpha, 1\right] = T[\id, 1]
            \]
            since $\beta = \alpha w$ and $w \in KZ$. Similarly, under the second map
            \[
                T_{-1, 0}[[\id, 1]] = \sum_{\lambda \in I_1}\left[\left[\begin{pmatrix}p & \lambda \\ 0 & 1\end{pmatrix}, 1\right]\right] \mapsto \sum_{\lambda \in I_1}\left[\begin{pmatrix}p & \lambda \\ 0 & 1\end{pmatrix}, -X^{p - 1}\right] = T[\id, -X^{p - 1}]. \qedhere
            \]
        \end{proof}
    \begin{remark}
    It is well known that elements in $\KZind V_0$ are in bijection with functions on the vertices of $\Delta$ with finite support. The element $[g, 1]$ corresponds to the function that sends the vertex $g \Zp^2$ to $1$ and all other vertices to $0$. In this spirit, one may ask what is the pictorial interpretation of the map
    \begin{eqnarray*}
        \IZind 1 & \xrightarrow{} & \KZind V_0.
    \end{eqnarray*}
    In the language of functions on the oriented edges and vertices of the tree $\Delta$, this map sends functions supported on edges to functions supported on the origins of those edges. For instance, if $F$ is a function that maps an (oriented) edge $e = (v_1, v_2)$ to $1$ and all other (oriented) edges to $0$, then the map displayed above sends $F$ to the function that maps $v_1$ to $1$ and all other vertices to $0$. Moreover, a short computation using the formulas of $T_{1, 0}$ and $T_{1, 2}$ in \cite[Lemma 6]{BL95} shows that on an edge $e''$, the value
    \[
        (T_{1, 2}T_{1, 0}F)(e'') =
        \begin{cases}
            1 & \text{ if } o(e'') = o(e) \text{ and } e'' \neq e, \\
            0 & \text{ otherwise.}
        \end{cases}
    \]
    Since there are $p$ edges satisfying the conditions $o(e'') = o(e)$ and $e'' \neq e$, we see that the isomorphism above sends $T_{1, 2}T_{1, 0}F$ to $0$. Pictorially, $F$ is supported on the red edge and $T_{1, 2}T_{1, 0}F$ is supported on the blue edges in the following figure:
    \begin{center}
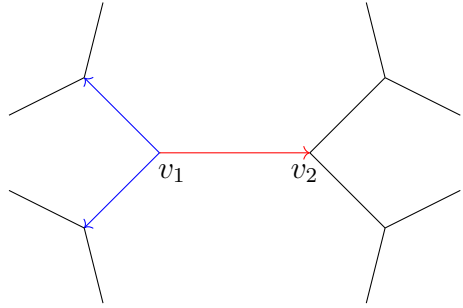

    \begin{tikzpicture}
        \draw[->, red] (-1, 0) -- (1, 0);
        \draw[black] (1, 0) -- (2, 1);
        \draw[black] (1, 0) -- (2, -1);
        \draw[black] (2, 1) -- (1.75, 2);
        \draw[black] (2, 1) -- (3, 0.5);
        \draw[black] (2, -1) -- (3, -0.5);
        \draw[black] (2, -1) -- (1.75, -2);
        \draw[<-, blue] (-2, 1) -- (-1, 0);
        \draw[<-, blue] (-2, -1) -- (-1, 0);
        \draw[black] (-2, 1) -- (-1.75, 2);
        \draw[black] (-2, 1) -- (-3, 0.5);
        \draw[black] (-2, -1) -- (-3, -0.5);
        \draw[black] (-2, -1) -- (-1.75, -2);
        \filldraw[black] (-1, 0) circle node[anchor=north]{\>\> $v_1$};
        \filldraw[black] (1, 0) circle node[anchor=north]{\!\!\!\! $v_2$};
    \end{tikzpicture}
    
    \captionof{figure}{The tree for $p = 2$.}
    \end{center}
    \end{remark}
\section{Proof of Theorem~\ref{Iwahori mod p LLC}}
In this section, we use Theorem~\ref{Comparison theorem in commutative Hecke algebra} and Theorem~\ref{Comparison theorem in the non-commutative Hecke algebra case} to rewrite Breuil's mod $p$ LLC as

\begin{theorem}
    For $r \in \{0, \ldots, p - 1\}$, $\lambda \in \brFp$ and a smooth character $\eta: \Qp^* \to \brFp^*$, define
    \begin{itemize}
        \item If $\lambda = 0:$ 
            \[
                (\ind \omega_2^{r + 1}) \otimes \eta \xleftrightarrow{\quad \quad} \frac{\IZind d^r}{(T_{-1, 0} + \delta_{r, p - 1}T_{1, 0}) + (T_{1, 2} + \delta_{r, 0}T_{1, 0})} \otimes \eta
            \]
        \item If $\lambda \neq 0:$
             \begin{eqnarray*}
                (\mu_{\lambda}\omega^{r + 1} \oplus \mu_{\lambda^{-1}})\otimes \eta \!\!\!\! & \xleftrightarrow{\quad \quad} \!\!\!\! & \left(\frac{\IZind d^r}{(T_{-1, 0} + \delta_{r, p - 1}T_{1, 0}) + (T_{1, 2} + \delta_{r, 0}T_{1, 0} - \lambda)} \otimes \eta\right)^{\!\!\rmss} \!\!\oplus \\
                && \left(\frac{\IZind d^{[p - 3 - r]}}{(T_{-1, 0} + \delta_{[p - 3 - r], p - 1}T_{1, 0}) + (T_{1, 2} + \delta_{[p - 3 - r], 0}T_{1, 0} - \lambda^{-1})}\otimes \eta\omega^{r + 1}\right)^{\!\!\rmss}\!\!\!.
            \end{eqnarray*}
    \end{itemize}
\end{theorem}
\begin{proof}
    We only have to prove that for $r \in \{0, 1, \ldots, p - 1\}$, $\lambda \in \brFp$ and a smooth character $\eta : \Qp^* \to \brFp^*$,
    \[
        \pi(r, \lambda, \eta) \simeq \frac{\IZind d^r}{(T_{-1, 0} + \delta_{r, p - 1}T_{1, 0}) + (T_{1, 2} + \delta_{r, 0}T_{1, 0} - \lambda)} \otimes \eta.
    \]
    
    First assume that $r \in \{1, 2, \ldots, p - 2\}$. Then using Theorem~\ref{Comparison theorem in commutative Hecke algebra}, there is an isomorphism
    \[
        \KZind V_{r} \simeq \frac{\IZind d^r}{(T_{-1, 0})}
    \]
    which sends the Hecke operator $T$ to $T_{1, 2}$. Therefore we see that
    \[
        \pi(r, \lambda, \eta) = \frac{\KZind V_{r}}{T - \lambda}\otimes \eta \simeq \frac{\IZind d^r}{(T_{-1, 0}) + (T_{1, 2} - \lambda)}\otimes \eta.
    \]

    Next, assume that $r = 0$. Using Theorem~\ref{Comparison theorem in the non-commutative Hecke algebra case}, there is an isomorphism
    \[
        \KZind V_0 \simeq \frac{\IZind 1}{(T_{1, 2}T_{1, 0})}
    \]
    under which the Hecke operator $T$ corresponds to $T_{-1, 0} + T_{1, 0}$. Next, using Remark~\ref{Order 2 automorphism}, we see that there is an isomorphism
    \[
        \frac{\IZind 1}{(T_{1, 2}T_{1, 0})} \simeq \frac{\IZind 1}{(T_{-1, 0})}
    \]
    under which $T_{-1, 0} + T_{1, 0}$ on the left corresponds to $T_{1, 2} + T_{1, 0}$ on the right. Therefore
    \[
        \pi(0, \lambda, \eta) = \frac{\KZind V_0}{(T - \lambda)}\otimes \eta \simeq \frac{\IZind 1}{(T_{-1, 0}) + (T_{1, 2} + T_{1, 0} - \lambda)} \otimes \eta.
    \]
    
    Finally assume that $r = p - 1$. Using Theorem $\ref{Comparison theorem in the non-commutative Hecke algebra case}$, there is an isomorphism
    \[
        \KZind V_{p - 1} \simeq \frac{\IZind 1}{(1 + T_{1, 2}T_{1, 0})}
    \]
    under which $T$ corresponds to $T_{-1, 0}$. Again using Remark~\ref{Order 2 automorphism} as above we have
    \[
        \frac{\IZind 1}{(1 + T_{1, 2}T_{1, 0})} \simeq \frac{\IZind 1}{(T_{-1, 0} + T_{1, 0})}
    \]
    under which $T_{-1, 0}$ on the left corresponds to $T_{1, 2}$ on the right. Therefore
    \[
        \pi(p - 1, \lambda, \eta) = \frac{\KZind V_{p - 1}}{(T - \lambda)}\otimes \eta \simeq \frac{\IZind 1}{(T_{-1, 0} + T_{1, 0}) + (T_{1, 2} - \lambda)}\otimes \eta. \qedhere
    \]
\end{proof}

{\noindent \bf Acknowledgements.}
I thank E. Ghate for assigning this problem and for many discussions. I also thank A. Jana for discussions on Iwahori-Hecke operators.


{\footnotesize
\begin{thebibliography}{1000000}
\bibitem[AB15]{AB15}
U. K. Anandavardhanan, Gautam Borisagar, \textit{Iwahori–Hecke model for supersingular representations of $GL_2(\Qp)$.} Journal of Algebra \textbf{423} (2015), 1--27.

\bibitem[AJ21]{AJ21}
U. K. Anandavardhanan, Arindam Jana, \textit{Iwahori–Hecke model for mod $p$ representations of $\mathrm{GL}(2,F)$.} Pacific J. Math. \textbf{315} (2021), no.2, 255–-283.

\bibitem[Ars21]{Ars21}
Bodan Arsovski, \textit{On the reductions of certain two-dimensional crystalline representations.} Doc. Math. \textbf{26} (2021), 1929–-1979.

\bibitem[BL94]{BL94}
Laure Barthel, Ron Livné, \textit{Irreducible modular representations of $\mathrm{GL}_2$ of a local field.} Duke Math. J.
\textbf{75} (1994), no. 2, 261-–292.

\bibitem[BL95]{BL95}
Laure Barthel, Ron Livné, \textit{Modular representations of $\mathrm{GL}_2$ of a local field: the ordinary, unramified case.} J. Number Theory \textbf{55} (1995), no. 1, 1–-27.

\bibitem[Ber10]{Ber10}
Laurent Berger, \textit{Repr\'esentations modulaires de $\GLQp$ et repr\'esentations galoisiennes de dimension 2.} Ast\'erisque \textbf{330} (2010), 263–-279.

\bibitem[BG15]{BG15}
Shalini Bhattacharya, Eknath Ghate, \textit{Reductions of Galois representations for slopes in (1,2).} Doc. Math. \textbf{20} (2015), 943–-987.

\bibitem[BGR18]{BGR18}
Shalini Bhattacharya, Eknath Ghate, Sandra Rozensztajn, \textit{Reductions of Galois representations of slope 1.} J. Algebra \textbf{508} (2018), 98–-156.

\bibitem[Bre03a]{Bre03a}
Christophe Breuil, \textit{Sur quelques repr\'esentations modulaires et $p$-adiques de $\GLQp$ I.} Compositio Math. \textbf{138} (2003), 165--188.

\bibitem[Bre03b]{Bre03b}
Christophe Breuil, \textit{Sur quelques repr\'esentations modulaires et $p$-adiques de $\GLQp$ II.} J. Inst. Math. Jussieu \textbf{2} (2003), 1--36.

\bibitem[Bre10]{Bre10}
Christophe Breuil, \textit{The emerging p-adic Langlands programme.} Proceedings of I.C.M. (2010), Vol. II, 203--230.

\bibitem[BG09]{BG09}
Kevin Buzzard and Toby Gee, \textit{Explicit reduction modulo p of certain 2-dimensional crystalline representations.} Int. Math. Res. Not. IMRN (2009), no.12, 2303–-2317.

\bibitem[BG13]{BG13}
Kevin Buzzard and Toby Gee, \textit{Explicit reduction modulo p of certain 2-dimensional crystalline representations, II.} Bull. Lond. Math. Soc. \textbf{45} (2013), no. 4, 779--788.

\bibitem[CG23]{CG23}
Anand Chitrao, Eknath Ghate, \textit{Reductions of semi-stable representations using the Iwahori mod $p$ Local Langlands Correspondence.} (2023), Preprint.

\bibitem[Col10]{Col10}
Pierre Colmez, \textit{Repr\'esentations de $\GL_2(\Qp)$ et $(\varphi,\Gamma)$-modules.} Ast\'erisque (2010), no.330, 281–-509.

\bibitem[GG15]{GG15}
Abhik Ganguli, Eknath Ghate, \textit{Reductions of Galois representations via the mod $p$ local Langlands correspondence.} J. Number Theory \textbf{147} (2015), 250–-286.

\bibitem[Gha21]{Gha21}
Eknath Ghate, \textit{A zig-zag conjecture and local constancy for Galois
representations.} RIMS K\^oky\^uroku Bessatsu \textbf{B86} (2021), 249--268.

\bibitem[Gha22]{Gha22}
Eknath Ghate, \textit{Zig-zag holds on inertia for large weights.} Preprint, \url{https://arxiv.org/abs/2211.12114}.

\bibitem[GJ23]{GJ23}
Eknath Ghate, Arindam Jana, \textit{Modular induced representations of $\GL_2(\mathbb{F}_q)$ as cokernels of symmetric power representations.} Preprint \href{https://arxiv.org/abs/2308.10246}{https://arxiv.org/abs/2308.10246}.

\bibitem[GR23]{GR23}
Eknath Ghate, Vivek Rai, \textit{Reductions of Galois representations of $\frac{3}{2}$.} Kyoto J. Math. (2023), to appear.

\bibitem[Jan23]{Jan23}
Arindam Jana, \textit{The pro-$p$-Iwahori invariants of the universal quotient representation for $\GL_2(F)$.} Int. J. Number Theory \textbf{19} (2023), no.8, 1853–-1879. 

\bibitem[NP19]{NP19}
Enno Nagel, Aftab Pande, \textit{Reductions of modular Galois representations of Slope $(2,3)$.} Preprint, \href{https://arxiv.org/abs/1801.08820}{https://arxiv.org/abs/1801.08820}.

\end{thebibliography}
}
\end{document}